\documentclass[12pt]{article}
\usepackage{a4}

\usepackage{amsthm}
\usepackage{amsmath}
\usepackage{enumitem}
\setlist{noitemsep}
\setlist[enumerate,1]{label=(\arabic*)}
\usepackage{graphicx}
\usepackage[english]{babel}
\usepackage{amsfonts}
\usepackage{pictex}
\usepackage{amssymb}

\newtheorem{theorem}{Theorem}
\newtheorem{obs}[theorem]{Observation}

\newtheorem{lemma}[theorem]{Lemma}

\newtheorem{claim}{Claim}

\newcommand*{\claimproof}{Proof of the Claim}
\newenvironment{proofcl}[1][\claimproof]{\begin{proof}[#1]
}{\end{proof}}

\newcommand{\claw}{K_{1,3}}
\newcommand{\bullp}{B_{1,p}}

\newcommand{\sm}{\setminus}

\newcommand{\FF}{\mathcal F}
\newcommand{\GG}{\mathcal G}
\newcommand{\HH}{\mathcal H}
\newcommand{\XX}{\mathcal X}

\title{Forbidden induced subgraphs for perfectness of claw-free graphs 
       of independence number at least~4}

\author{Christoph Brause$^{1}$     \and 
        Trung Duy Doan$^{2}$      \and
        P\v{r}emysl Holub$^{3}$  \and
        Adam Kabela$^{3}$   \and
        Zden\v{e}k Ryj\'a\v{c}ek$^{3}$  \and
        Ingo Schiermeyer$^{1}$ \and
        Petr Vr\'ana$^{3}$ \\ }

\date{}
\begin{document}
\maketitle
\footnotetext[1]{Institut f\"{u}r Diskrete Mathematik und Algebra, 
    Technische Universit\"{a}t Bergakademie Freiberg, Germany.
    E-mails: brause@math.tu-freiberg.de, Ingo.Schiermeyer@tu-freiberg.de.}
\footnotetext[2]{School of Applied Mathematics and Informatics
	Hanoi, University of Science and Technology, Vietnam.
    E-mail: trungdoanduy@gmail.com.}
\footnotetext[3]{Department of Mathematics and European Centre of Excellence NTIS,
    University of West Bohemia, Czech Republic.
    E-mails: \{holubpre, kabela, ryjacek, vranap\}@kma.zcu.cz.}
\begin{abstract}
For every graph $X$, we consider the class of all connected $\{\claw,X\}$-free 
graphs which are distinct from an odd cycle and have independence number at 
least $4$, and we show that all graphs in the class are perfect if and only if
$X$ is an induced subgraph of some of
$P_6$, $K_1 \cup P_5$, $2P_3$, $Z_2$ or $K_1 \cup Z_1$.
Furthermore, for $X$ chosen as $2K_1 \cup K_3$, we list all eight imperfect 
graphs belonging to the class; and for every other choice of $X$, we show that 
there are infinitely many such graphs.
In addition, for $X$ chosen as $B_{1,2}$, we describe the structure of all 
imperfect graphs in the class.

\end{abstract}
%

\section{Introduction}
\label{sec:intro}

We consider finite, simple, undirected graphs, and we refer to \cite{BM08} 
for terminology and notation not defined here. 
We let $N(x)$ denote the set of all vertices adjacent to vertex $x$ in a given graph.
Considering a graph $G$ and a set $S$ of its vertices,
we recall that a subgraph of $G$ 
\emph{induced by $S$} is simply the graph obtained from $G$
by removing all vertices of $V(G) \sm S$.
We say that a graph $H$ is an \emph{induced subgraph} of $G$ if
there is a set of vertices of $G$ which induces a graph isomorphic to $H$.
Given a family $\HH$ of graphs and a graph $G$, we say that $G$ is 
\emph{$\HH$-free} if $G$ contains no graph from $\HH$ as an induced subgraph.
In this context, the graphs of $\HH$ are referred to as 
\emph{forbidden subgraphs}. 
We emphasise that the studied forbidden subgraphs are not necessarily connected.
We let $H_1 \cup H_2$ denote the disjoint union of graphs $H_1$ and $H_2$,
and let $kH$ denote the disjoint union of $k$ copies of a graph~$H$.
A cycle of length at least $4$ is called a {\it hole},
and a graph whose complement is a cycle of length at least $4$
is called an {\it antihole}.
A hole (antihole) is {\it odd} if it has an odd number of vertices.
(We usually talk about holes and antiholes as induced subgraphs.)
We recall that a graph is {\it $k$-colourable}
if each of its vertices can be coloured with one of $k$ colours
so that adjacent vertices are assigned distinct colours.
The smallest integer $k$ such that a given graph $G$ is $k$-colourable
is called the {\it chromatic number} of $G$, denoted by $\chi(G).$
We let $\omega(G)$ denote the \emph{clique number} of $G$,
that is, the order of a maximum complete subgraph of $G$.
(Clearly, $\chi(G) \geq \omega(G)$ for every graph $G$.)
%
A graph $G$ is {\it perfect} if $\chi(G')=\omega(G')$
for every induced subgraph $G'$ of $G.$

\medskip

Studying connected $\claw$-free graphs with independence number at least $3$, 
Chudnovsky and Seymour~\cite{CS} showed that their chromatic number can be at 
most twice as large as the clique number (and they also presented an infinite 
family of such graphs whose chromatic number is almost this large).
Considering a $3K_1$-free graph $G$ (clearly, $\claw$-free and of independence 
at most $2$), we recall that its chromatic number is at least 
$\frac{1}{2}|V(G)|$; and for some such graphs, $|V(G)|$ has order of magnitude 
$\frac{\omega(G)^2}{\log \omega(G)}$ (by a result of Kim~\cite{Kim} on Ramsey 
numbers).

While we are focused mainly on $\claw$-free graphs, we should say that relating 
forbidden induced subgraphs and colourings is a classical topic in graph theory.
Numerous results are known and, naturally, stronger colouring properties can be 
obtained when considering a pair (or larger set) of forbidden induced subgraphs
(for instance, see survey papers~\cite{GJPS17,RS04,RS19}).

We investigate restricting the class of $\claw$-free graphs by additional 
constraints (in particular, by different choices of an additional forbidden 
induced subgraph $X$) so that the resulting class consists of perfect graphs.

To this end, we will use the classical result on perfect graphs by Chudnovsky et al.~\cite{ChRST06}
as the main tool. 

\begin{theorem}[Chudnovsky et al.~\cite{ChRST06}]
\label{tSPGT}
A graph is perfect if and only if it contains neither an odd hole 
nor an odd antihole as an induced subgraph.
\end{theorem} 

We also use the following lemma due to Ben Rebea~\cite{BR}
(see also~\cite{ChS88,Fo93}).

\begin{lemma}[Ben Rebea~\cite{BR}]
\label{lBR}
Let $G$ be a connected $\claw$-free graph
with independence number at least $3$.
If $G$ contains an induced odd antihole,
then $G$ contains an induced $C_5$.
\end{lemma}

\medskip

Our former investigation of connected $\{ \claw,X \}$-free graphs
(of independence at least $3$) resulted in the following characterisations.

\begin{theorem}[Brause et al.~\cite{alpha3}]
\label{thmA-characterization-alpha_3}
Let $X$ be a graph and
$\mathcal{G}$ be the class of all connected $\{K_{1,3},X\}$-free graphs
which are distinct from an odd cycle.
Then the following statements are satisfied.
\begin{itemize}[topsep=5pt, partopsep=5pt]
\item
If $X$ is an induced subgraph of $Z_1$ or $P_4$,
then all graphs of $\mathcal{G}$ are perfect.
\item
Otherwise, there are infinitely many graphs of $\mathcal{G}$
whose chromatic number is greater than the clique number.
\end{itemize}
Furthermore, the following are satisfied for the class $\mathcal{G}'$ of
all graphs of $\mathcal{G}$ whose independence number is at least $3$.
\begin{itemize}[topsep=5pt, partopsep=5pt]
\item
If $X$ is an induced subgraph of $Z_2$ or of $P_5$,
then all graphs of $\mathcal{G}'$ are perfect.
\item
Otherwise, there are infinitely many graphs of $\mathcal{G}'$
whose chromatic number is greater than the clique number.
\end{itemize}
\end{theorem}

\begin{figure}[ht]
\centering
\includegraphics[scale=0.7]{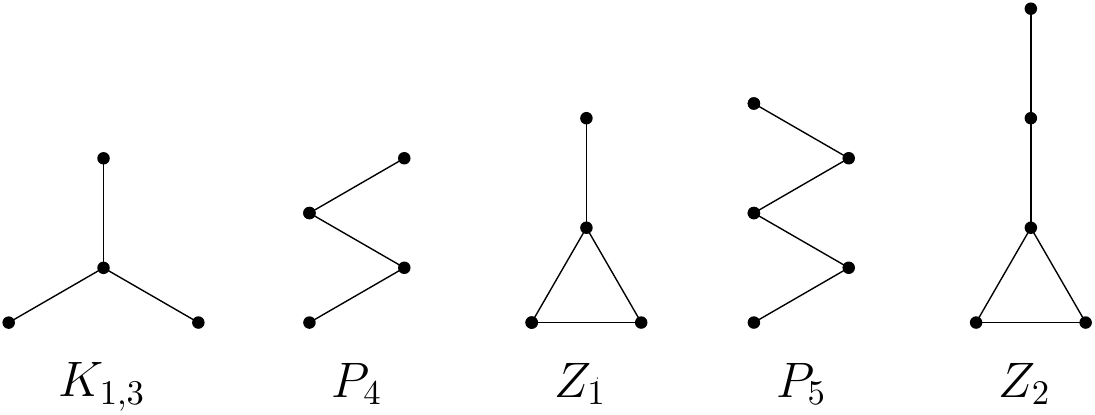}
    \caption{The graphs $\claw$, $P_4$, $Z_1$, $P_5$ and $Z_2$ 
    (considered in Theorem~\ref{thmA-characterization-alpha_3}).}
    \label{figClawPsZs}
\end{figure}

The dichotomic nature of Theorem~\ref{thmA-characterization-alpha_3}
(for graphs with independence number at least $2$ and at least $3$)
raises a question on the nature of an analogous statement
for graphs of independence number at least $4$.  
Motivated by this question, we look one step further in
this direction and investigate perfectness of these graphs.

\section{Main result}
\label{sec-main}

In this section, we answer the question motivated by the nature of Theorem~\ref{thmA-characterization-alpha_3}.
The dichotomic character of Theorem~\ref{thmA-characterization-alpha_3} does not extend to
$\{K_{1,3},X\}$-free graphs with independence number at least $4$.
We show a full characterization and describe the finitely many exceptions,
which are given by one of the forbidden pairs. 
The main result of the present note is as follows.

\begin{theorem}
\label{t1}
Let $X$ be a graph
and $\GG$ be the class of all connected $\{ \claw,X \}$-free graphs
which are distinct from an odd cycle and have independence number at least $4$.
Let $\XX$ be the set of graphs which consists of
$P_6$, $K_1 \cup P_5$, $2P_3$, $Z_2$, $K_1 \cup Z_1$
and all their induced subgraphs.
The following statements are satisfied.
\begin{enumerate}[topsep=5pt, partopsep=5pt]
\item
If $X$ belongs to $\XX$,
then all graphs of $\GG$ are perfect.
\item
If $X$ is $2K_1 \cup K_3$, then the only imperfect graphs of $\GG$
are the graphs $E_1, \ldots, E_8$, depicted in Figure~\ref{figExceptions}.
\item
If $X$ does not belong to $\XX \cup \{ 2K_1 \cup K_3 \}$,
then $\GG$ contains infinitely many imperfect graphs.
\end{enumerate}
\end{theorem}

\begin{figure}[ht]
\centering
\includegraphics[scale=0.7]{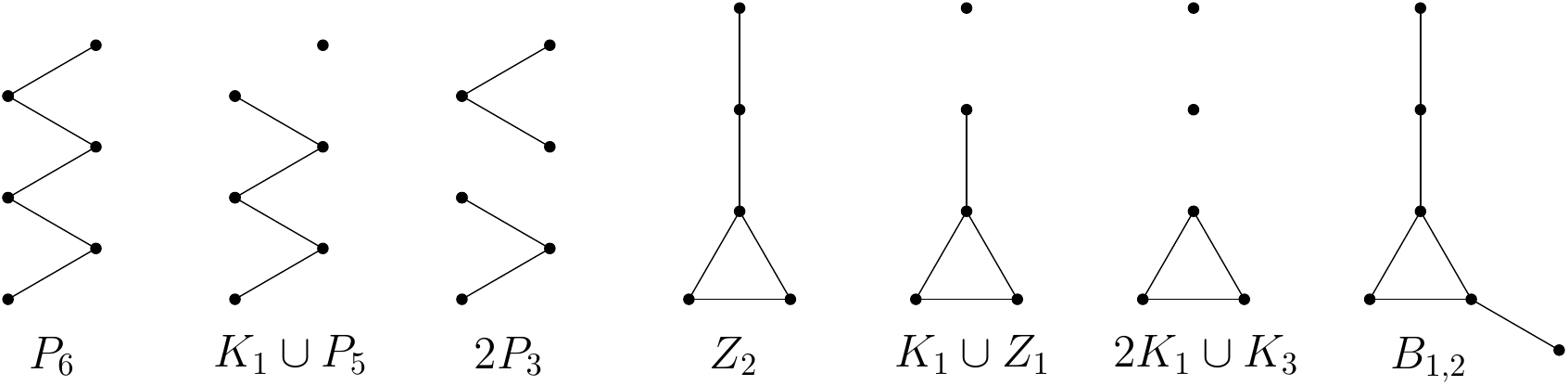}
    \caption{The graphs
    $P_6$, $K_1 \cup P_5$, $2P_3$, $Z_2$, $K_1 \cup Z_1$, $2K_1 \cup K_3$ and
    $B_{1,2}$
    (considered in Theorems~\ref{t1} and~\ref{thm-bull-exceptions}).}
    \label{figChoicesOfX}
\end{figure}

\begin{figure}[ht]
\centering
\includegraphics[scale=0.7]{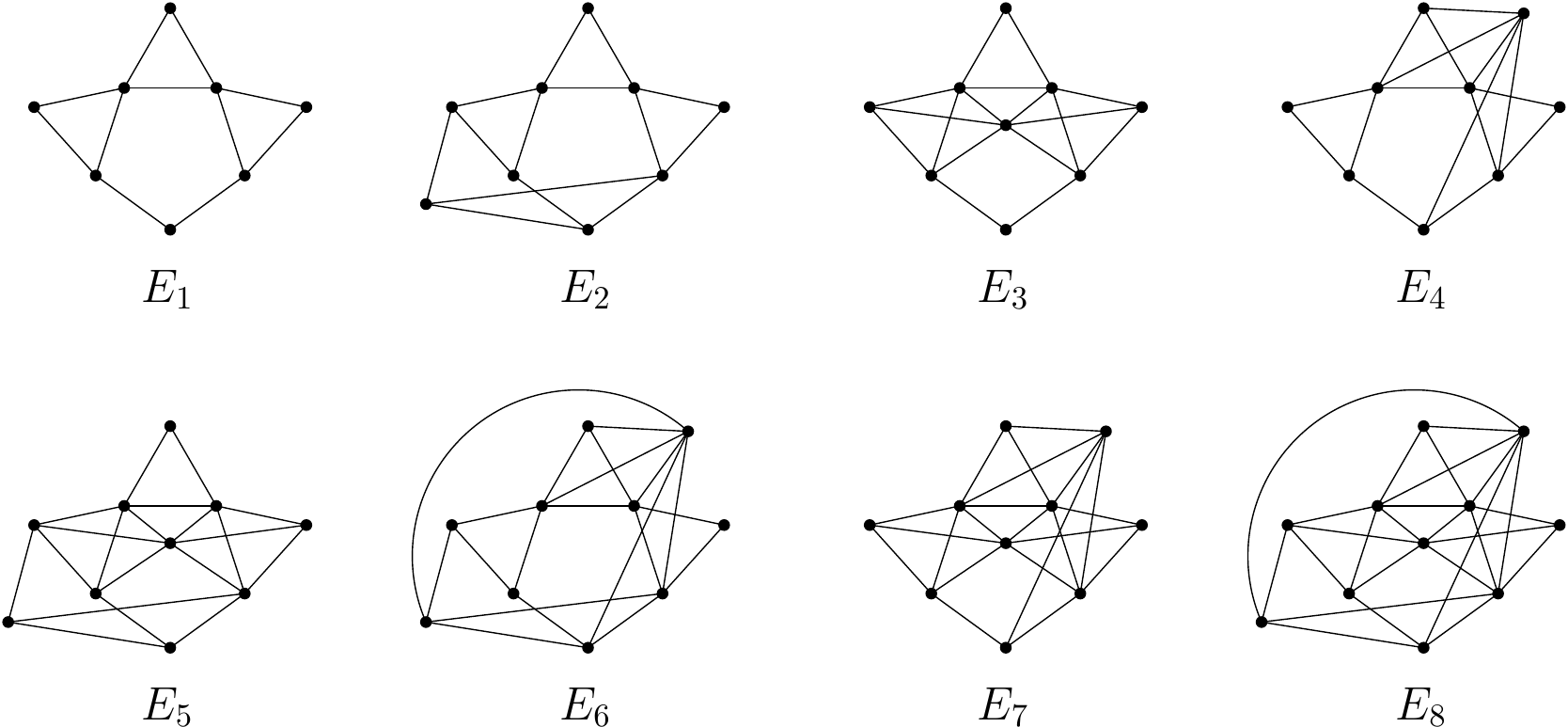}
    \caption{The graphs $E_1, \ldots, E_8$.}
    \label{figExceptions}
\end{figure}

We note that the assumption of being distinct from an odd cycle
is satisfied trivially for particular choices of $X$,
and that similar characterisations follow for the case when the graphs 
considered are not necessarily distinct from an odd cycle.
(This concerns choosing $X$ as an induced subgraph of $P_4$ or $P_5$
in the respective parts of Theorem~\ref{thmA-characterization-alpha_3},
and as an induced subgraph of $P_6$, $K_1 \cup P_5$ or $2P_3$ in 
Theorem~\ref{t1}.)

We also note that other choices of the graph $X$ (in item (3) of Theorem~\ref{t1}) 
can still admit a `nice' description of all (infinitely many) imperfect graphs in the class~$\GG$.
This fact is illustrated on the example $X=B_{1,2}$ (see Figure~\ref{figChoicesOfX})
by proving Theorem~\ref{thm-bull-exceptions} in Section~\ref{sec:concl}.

\medskip

In order to prove Theorem~\ref{t1}, we will show three structural lemmas on $\claw$-free graphs.

\begin{lemma}
\label{l5}
Let $G$ be a connected $\claw$-free graph
with independence number at least $4$, and
$H_1, \ldots, H_7$ be the graphs depicted in Figure~\ref{figH}.
The following statements are satisfied.
\begin{enumerate}[topsep=5pt, partopsep=5pt]
\item
If $G$ is $H_1$-free,
then it is $C_7$-free.
\item
If $G$ is $\{H_2, \ldots, H_7 \}$-free,
then it is $C_5$-free. 
\end{enumerate}
\end{lemma}

\begin{figure}[ht]
\centering
\includegraphics[scale=0.75]{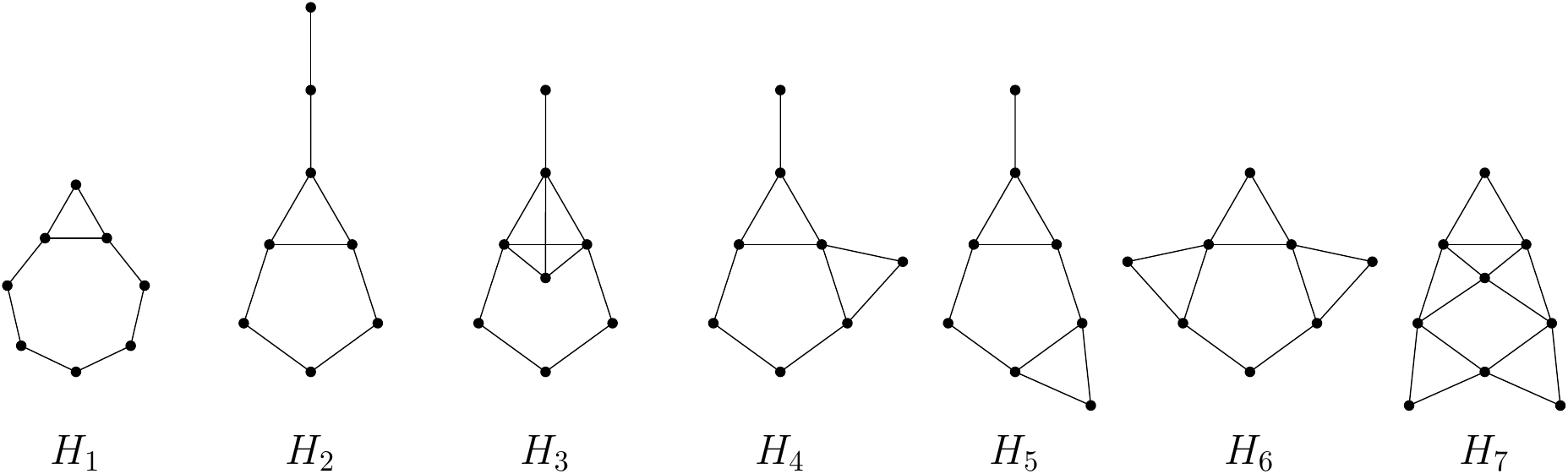}
    \caption{The graphs $H_1, \ldots, H_7$.
    We note that $H_6$ is isomorphic to the graph $E_1$,
    depicted in Figure~\ref{figExceptions}.}
    \label{figH}
\end{figure}

\begin{lemma}
\label{l6}
If $G$ is a connected $\{ \claw, 2K_1 \cup K_3 \}$-free graph with independence 
number at least~$4$ which is distinct from the graphs $E_1, \ldots, E_8$ 
(depicted in Figure~\ref{figExceptions}), then $G$ is $\{ C_5, C_7 \}$-free.
\end{lemma}

\begin{lemma}
\label{l7}
Let $\GG$ be the class defined in Theorem~\ref{t1},
and let $X$ be a graph such that $X$ is not an induced subgraph of any of 
$P_6$, $K_1 \cup P_5$, $2P_3$, $Z_2$, $K_1 \cup Z_1$, $2K_1 \cup K_3$.
Then infinitely many graphs of $\GG$ contain an induced odd hole.
\end{lemma}

We also recall the following fact observed in~\cite{alpha3}.

\begin{obs}[Brause et al.~\cite{alpha3}]
\label{o8}
Let $G$ be a $\claw$-free graph and $C$ be a set of its vertices
such that $C$ induces a cycle of length at least $5$.
If $x$ is a vertex that does not belong to $C$ but is adjacent to a vertex of $C$, then 
$N(x) \cap C$ induces $K_2$ or $P_3$ or $P_4$,
or (in case $|C| = 5$) it can induce $C_5$,
or (in case $|C| \geq 6$) it can induce $2K_2$. 
\end{obs}

The proofs of Lemmas~\ref{l5}, \ref{l6} and~\ref{l7} are included below.
Assuming the lemmas are true, we prove Theorem~\ref{t1}.

\begin{proof}[Proof of Theorem~\ref{t1}]
We assume that $X$ belongs to $\XX \cup \{ 2K_1 \cup K_3 \}$,
and prove statements (1) and (2).
First, we show that $G$ is $\{ C_9, C_{11}, \ldots \}$-free.
For the sake of a contradiction, we suppose that $G$ contains an induced $C_{\ell}$
(where $\ell \geq 9$ and $\ell$ is odd).  
Clearly, $G$ contains each of the graphs $P_6, K_1 \cup P_5, 2P_3$
as an induced subgraph.
Since $G$ is connected and distinct from $C_{\ell}$,
it contains an additional vertex which is adjacent to a vertex of this $C_{\ell}$.
Using Observation~\ref{o8}, we conclude that $G$ also contains  $Z_2$, $K_1 \cup Z_1$, and $2K_1 \cup K_3$ as induced subgraphs, a contradiction.

Next, we show that $G$ is $\{ C_5, C_7 \}$-free.
For the case when $X$ belongs to $\XX$,
we observe that each of the graphs $H_1, \ldots, H_7$
(depicted in Figure~\ref{figH}) contains $X$ as an induced subgraph.
In particular, $G$ is $\{ H_1, \ldots, H_7 \}$-free,
and thus $\{ C_5, C_7 \}$-free by Lemma~\ref{l5}.
For the case when $X$ is $2K_1 \cup K_3$,
we conclude that $G$ is $\{ C_5, C_7 \}$-free by Lemma~\ref{l6}.

In particular, we can now use the fact that $G$ is $C_5$-free as follows.
Since $G$ satisfies the assumptions of Lemma~\ref{lBR},
we get that $G$ cannot contain an odd antihole as an induced subgraph
(if $G$ contained an induced odd antihole, then it would contain
induced $C_5$, contradicting the fact that $G$ is $C_5$-free).

Consequently, $G$ contains neither an odd hole 
nor an odd antihole as an induced subgraph,
and thus $G$ is a perfect graph by Theorem~\ref{tSPGT}.

To conclude the proof, we observe that statement (3) follows by Lemma~\ref{l7}. 
\end{proof}

In the remainder of the present section,
we prove Lemmas~\ref{l5},~\ref{l6}, and~\ref{l7}.

\begin{proof}[Proof of Lemma~\ref{l5}]
We prove the lemma by considering a minimal counterexample.
In particular, for each of the two statements,
we consider a graph $G$ which satisfies the assumptions of the statement
and contains an induced $C_\ell$
(where $\ell = 7, 5$ for statement (1), (2), respectively)
and, subject to these properties,
has a minimal number of vertices.

We let $C$ be a set of vertices inducing $C_\ell$ in $G$,
and $N(C)$ be the set of all vertices not belonging to $C$
but adjacent to a vertex of $C$.
We let $I$ be a maximum independent set of $G$,
and $e$ be the number of edges going from $C$ to $I \sm C$.
Before proving the statements, we show three claims on basic properties of $G$.

\begin{claim} 
\label{cd2}
If $x$ is a vertex with a neighbour in $C$ and a neighbour outside $C \cup N(C)$,
then $N(x)\cap C$ induces $K_2$. 
\end{claim}
\begin{proofcl}[Proof of Claim~\ref{cd2}]
Since $G$ is $\claw$-free, the set $N(x)\cap C$ cannot contain two 
non-adjacent vertices, and $|N(x)\cap C| \neq 1$.
It follows that $N(x)\cap C$ induces~$K_2$.
\end{proofcl}
\begin{claim} 
\label{ctype3}
For every vertex $x$ of $G$, the set $N(x) \cap C$ does not induce $P_3$.
\end{claim}
\begin{proofcl}[Proof of Claim~\ref{ctype3}]
For the sake of a contradiction, we suppose that there is a vertex $x$
whose neighbours in $C$ induce $P_3$.
We let $y$ be the central vertex of this $P_3$,
and we consider the graphs $G - x$ and $G - y$.
We note that both considered graphs are connected
(since $G$ is $\claw$-free),
and both contain an induced $C_\ell$.
Furthermore, at least one of the considered graphs is of independence at least~$4$
(since $x$ and $y$ cannot both belong to $I$),
and we conclude that this graph contradicts the choice of $G$ as a minimal counterexample.
\end{proofcl}
\begin{claim} 
\label{cineq}
We have $e \leq 2\ell - 4|I \cap (C \cup N(C))| + 4|I \cap N(C)|$.
\end{claim}
\begin{proofcl}[Proof of Claim~\ref{cineq}]
We recall that $I$ is an independent set.
We consider a vertex $x$ of $C$ and discuss the number of its neighbours in $I \sm C$,
that is, the contribution to the quantity $e$. 
Clearly, $x$ is adjacent to at most two vertices of $I$
(since $G$ is $\claw$-free).
Furthermore, if $x$ is adjacent to a vertex of $I \cap C$,
then $x$ has at most one neighbour in $I \sm C$.
Similarly, if $x$ is adjacent to two vertices of $I \cap C$ or $x$ belongs to $I$,
then $x$ has no neighbour in $I \sm C$.
Consequently, we get that $e \leq 2\ell - 4|I \cap C|$,
and the desired inequality follows since
$|I \cap C| = |I \cap (C \cup N(C))| - |I \cap N(C)|$.
\end{proofcl}
We use Claims~\ref{cd2}, \ref{ctype3} and~\ref{cineq}, and show statements (1) and (2).
First, we consider a graph $G$ chosen as a minimal counterexample to statement (1),
and a set $C$ inducing $C_7$ in $G$.
We note that Claim~\ref{cd2} implies that every vertex of $V(G) \sm C$ is adjacent to a vertex of $C$
(since $G$ is connected and $H_1$-free).
In particular, we have $I \cap N(C) = I \sm C$.
Furthermore, every vertex of $V(G) \sm C$ has precisely four neighbours in $C$
(by combining the fact that $G$ is $H_1$-free together with Claim~\ref{ctype3} and Observation~\ref{o8}).
In particular, we consider the vertices of $I \sm C$,
and conclude that $e = 4|I \sm C|$.
On the other hand, Claim~\ref{cineq} yields that $e \leq 14 - 16 + 4|I \cap N(C)| = 4|I \sm C| - 2$, a contradiction.

Next, we consider a minimal counterexample $G$ for statement (2),
and a set $C$ inducing $C_5$.
We first show that every vertex of $V(G) \sm C$ is adjacent to a vertex of $C$.
For the sake of a contradiction,
we suppose that there is a vertex, say~$w$, which has no neighbour in $C$.
We observe that $w$ is precisely at distance two from $C$ 
(using Claim~\ref{cd2} and the facts that $G$ is connected and $H_2$-free).
Furthermore, the graph $G - w$ is connected and it contains an induced $C_5$,
and hence $w$ belongs to~$I$ and $|I| = 4$
(since $G$ is a minimal counterexample).
We let $u$ be a vertex adjacent to $w$ and to a vertex of $C$.
In particular, $u$ does not belong to $I$ and $G - u$ contains an induced $C_5$. 
We use the fact that $G$ is a minimal counterexample
and note that the graph $G - u$ is not connected,
and furthermore
$G - u$ has precisely two components one of which consists only of the vertex $w$
(since $G$ is $\claw$-free and minimal).
Consequently,
we observe that $u$ is the only vertex of $G$ whose neighbours in $C$ induce $K_2$
(since $G$ is $\{ H_3, H_4, H_5 \}$-free
and the graphs depicted in Figure~\ref{figNotH345} are not $\claw$-free). 
\begin{figure}[ht]
\centering
\includegraphics[scale=0.7]{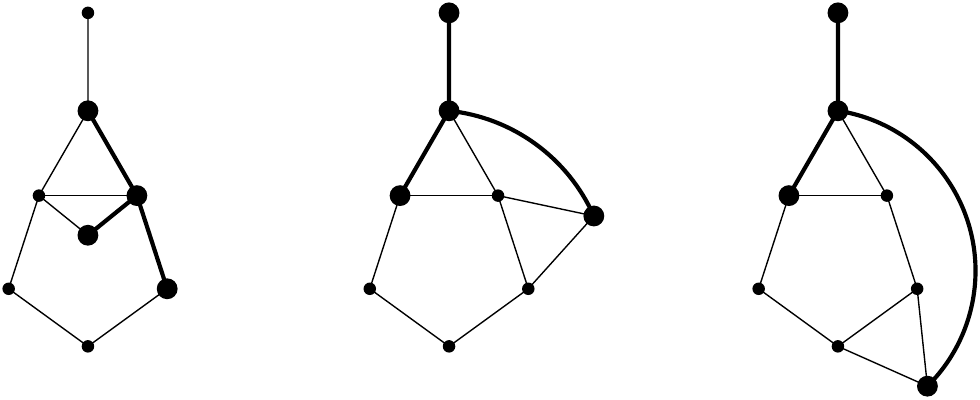}
    \caption{Possible graphs induced by $C \cup \{u,w,z\}$,
    where $z$ is another vertex whose neighbours in $C$ induce $K_2$,
    and distinct from $H_3, H_4$ and $H_5$.
    We note that the graphs are not $\claw$-free
    (induced copies of $\claw$ are depicted in bold).}
    \label{figNotH345}
\end{figure}
It follows that every vertex of $I \cap N(C)$ has at least four neighbours in~$C$ 
(by Claim~\ref{ctype3} and Observation~\ref{o8}),
and thus $e \geq 4|I \cap N(C)|$.
On the other hand,
we note that $w$ is the only vertex of $I$ which has no neighbour in $C$
(by Claim~\ref{cd2}).
Thus, we have $|I \cap (C \cup N(C))| = 3$,
and Claim~\ref{cineq} gives that
$e \leq 10 - 12 + 4|I \cap N(C)| = 4|I \cap N(C)| - 2$, a contradiction.

\begin{figure}[ht]
\centering
\includegraphics[scale=0.7]{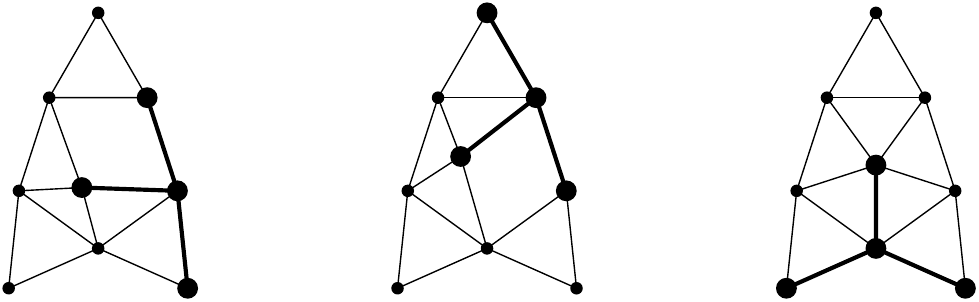}
    \caption{Possible graphs induced by $C \cup I$ which are distinct from $H_7$.
    Each graph contains an induced $\claw$
    (depicted in bold).}
    \label{figNotH7}
\end{figure}

Hence, every vertex of $V(G) \sm C$ is adjacent to a vertex of $C$.
In particular, every vertex of $V(G) \sm C$ belongs to $I$
(since $G$ is a minimal counterexample).
We recall that $G$ is $\{ H_6, H_7 \}$-free, and we observe that at most two vertices of $I \sm C$
have the property that its neighbourhood in $C$ induces $K_2$
(see Figure~\ref{figNotH7}).
Similarly as above, Claim~\ref{ctype3} and Observation~\ref{o8} imply that
$e \geq 4|I \sm C| - 4$.
However, Claim~\ref{cineq} yields that
$e \leq 10 - 16 + 4|I \sm C| = 4|I \sm C| - 6$,
a contradiction.
\end{proof}

\begin{proof}[Proof of Lemma~\ref{l6}]
We consider such graph $G$
and note that $G$ is $\{ H_1, \dots, H_5 \}$-free and $H_7$-free
(since it is $2K_1 \cup K_3$-free).

We suppose that 
$G$ is distinct from $E_1, \ldots, E_8$ (depicted in Figure~\ref{figExceptions}),
and show that $G$ is $H_6$-free.
For the sake of a contradiction, we suppose that $G$ contains a set $H$
of vertices inducing $H_6$.
We note that $G$ contains a vertex, say~$x$,
which does not belong to $H$ but is adjacent to a vertex of $H$
(since $G$ is connected and distinct from $E_1$, that is, $H_6$).
We discuss the adjacency of $x$ to the vertices of $H$
and show that there are essentially only three types of connecting $x$ to $H$
(see Figure~\ref{figH5andx}).
To this end,
we consider the set $I$ of four independent vertices of $H$
and the labelling of vertices of $H$ given in Figure~\ref{figH6}.
We note that $x$ is non-adjacent to at least two vertices of $I$
(since $G$ is $\claw$-free),
and we discuss the cases given by pairs of vertices of~$I$.
For each case, we use the assumption that $G$ is $\{ \claw, 2K_1 \cup K_3 \}$-free.
By symmetry, we need to consider four cases as follows.

\begin{figure}[ht]
\centering
\includegraphics[scale=0.75]{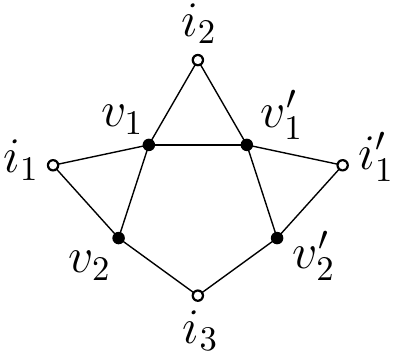}
    \caption{Labelling the vertices of $H_6$.
    The vertices of $I$ are labelled $i_1, i_1', i_2, i_3$.}
    \label{figH6}
\end{figure}

First, we show that if $x$ is adjacent to neither $i_2$ nor $i_3$,
then $x$ is adjacent to all remaining vertices of $H$
(this gives type $B$ as depicted in Figure~\ref{figH5andx}).
We use that the set $\{i_2, i_3, v_1, v_1', x \}$ cannot induce
$2K_1 \cup K_3$, and hence we can assume that $x$ is adjacent to $v_1$.
Using the edge $x v_1$, we get that $x$ is also adjacent to $i_1$
(by considering the graph induced by $\{i_1, i_2, v_1, x \}$).
Similarly, $x$ is adjacent to $v_2$.
Consequently, $x$ is adjacent to $i_1'$
(by considering $\{i_1, i_1', i_2, v_2, x \}$),
and similarly it is adjacent to $v_2'$.
Finally, it is adjacent to $v_1'$
(by considering $\{i_3, v_1', v_2', x \}$).

Second, we assume that $x$ is adjacent to neither $i_1'$ nor $i_2$.
Using the previous case, we can assume that $x$ is adjacent to $i_3$.
It follows that $x$ is adjacent to neither $v_1'$ nor $v_2$,
and consequently it is not adjacent to $v_1$.
Finally, $x$ is adjacent to $v_2'$,
and thus it is adjacent to $i_1$
(this gives type $A$).

Third, we suppose that $x$ is adjacent to neither $i_1$ nor $i_3$,
and we can assume that $x$ is adjacent to $i_2$.
We note that $x$ is adjacent to neither $v_2$ nor $v_1'$.
Since $x$ is not adjacent to $v_1'$,
it is adjacent to neither $v_1$ nor $v_2'$.
Hence, $x$ is not adjacent to~$v_2$.
A contradiction follows by considering 
$\{i_1, v_1, v_2, v_2', x \}$.

Fourth, we assume that $x$ is adjacent to neither $i_1$ nor $i_1'$,
and that $x$ is adjacent to $i_2$ and $i_3$.
In addition, we can assume that $x$ is adjacent to $v_2'$
(by considering $\{i_3, v_2, v_2', x \}$).
It follows that $x$ is not adjacent to $v_2$,
and hence adjacent to $v_1$, and thus adjacent to $v_1'$
(this gives type $C$).

With the three types on hand (as depicted in Figure~\ref{figH5andx}),
we continue the argument.
In particular, we note that 
$G$ contains at least two vertices which do not belong to $H$
(since $G$ is distinct from $E_2, E_3$ and $E_4$).

\begin{figure}[ht]
\centering
\includegraphics[scale=0.75]{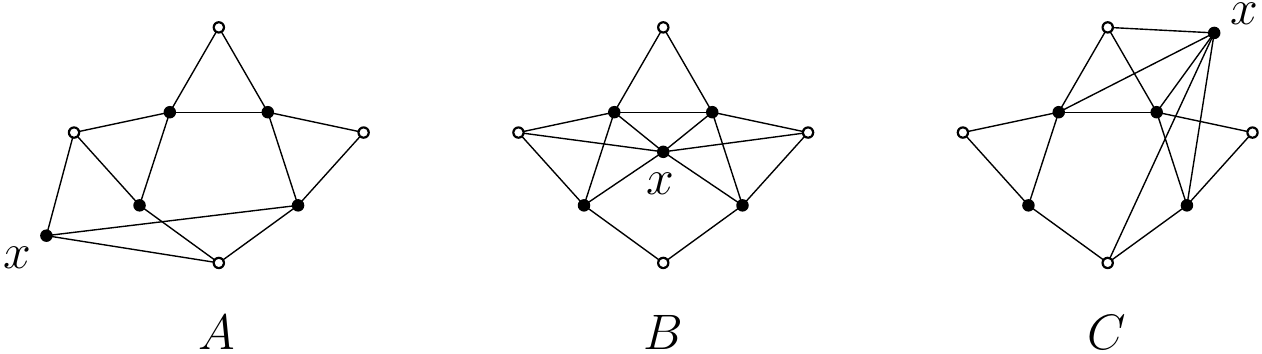}
    \caption{The only three types of connecting $x$ to $H_6$.
	The types are labelled $A, B$ and $C$.    
    The vertices of $I$ are depicted in white.}
    \label{figH5andx}
\end{figure}

In addition, we have that every vertex of $G$ is adjacent to a vertex of $H$
(since $G$ is connected and $\claw$-free, and two non-adjacent vertices of $H$
are adjacent to~$x$).
Hence, these three types (of connecting $x$) apply to every vertex of $V(G) \sm H$.
We consider a pair of such vertices,
and note that their neighbourhoods in $H$ cannot be the same 
(since $G$ is $\{ \claw, 2K_1 \cup K_3 \}$-free).
In particular if adding two vertices of type $A$, then
one has to be adjacent to $i_1, i_3, v_2'$
and the other to $i_1', i_3, v_2$.
We consider the possible graphs obtained by adding two vertices of type $A$
(there are two graphs to consider since the additional vertices may or may not be adjacent),
and we note that none of them is $\{ \claw, 2K_1 \cup K_3 \}$-free.
Similarly, we discuss all remaining cases and observe that
there are precisely three options of
connecting two vertices to $H$ (see Figure~\ref{figH5plus2}).

\begin{figure}[ht]
\includegraphics[scale=0.7]{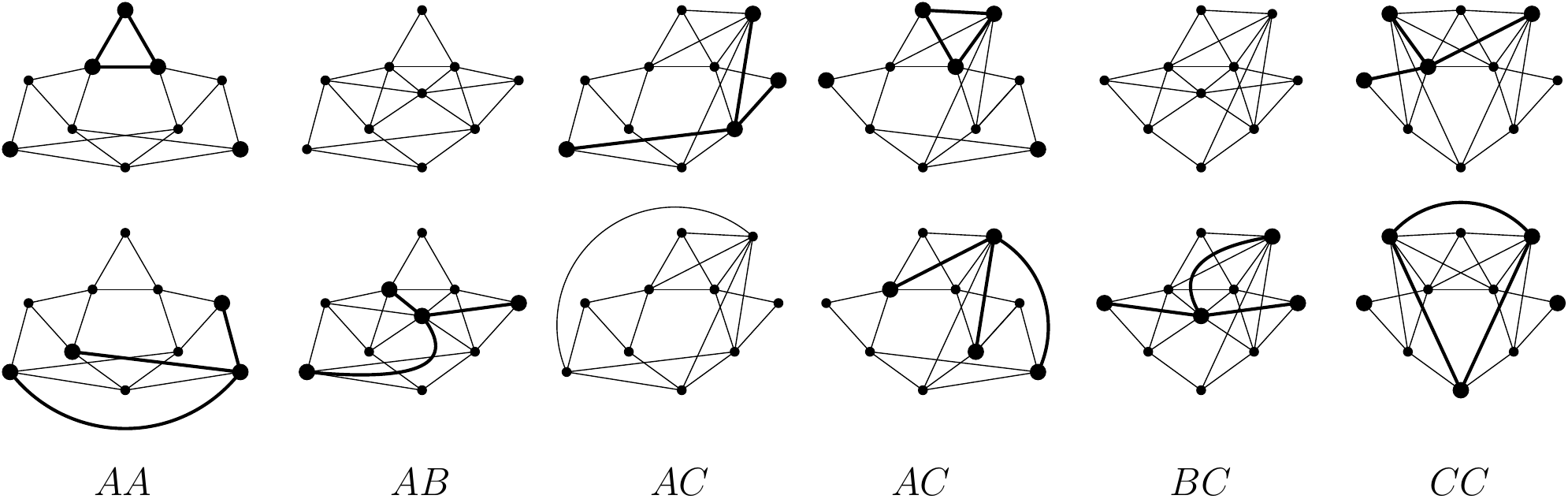}
    \caption{Graphs obtained by adding two vertices to $H_6$.
    The labelling indicates types of the additional vertices. 
	For each pair of the types, the additional vertices
	might be non-adjacent (top) or adjacent (bottom).
    Three of the graphs are $\{ \claw, 2K_1 \cup K_3 \}$-free.
    In the remaining graphs, induced $\claw$ or $2K_1 \cup K_3$ is
    highlighted.}
    \label{figH5plus2}
\end{figure}

Since $G$ is distinct from $E_5, E_6$ and $E_7$, 
it contains at least three vertices which do not belong to $H$.
We consider pairs of such vertices and the above discussion,
and we conclude that $G$ contains precisely three such vertices 
and, in fact, $G$ is exactly $E_8$, a contradiction.
Hence, we can assume that $G$ is $H_6$-free.

Consequently, $G$ is $\{ H_1, \ldots, H_7 \}$-free,
and thus $\{ C_5, C_7 \}$-free by Lemma~\ref{l5}.
\end{proof}

\begin{figure}[ht]
\centering
\includegraphics[scale=0.75]{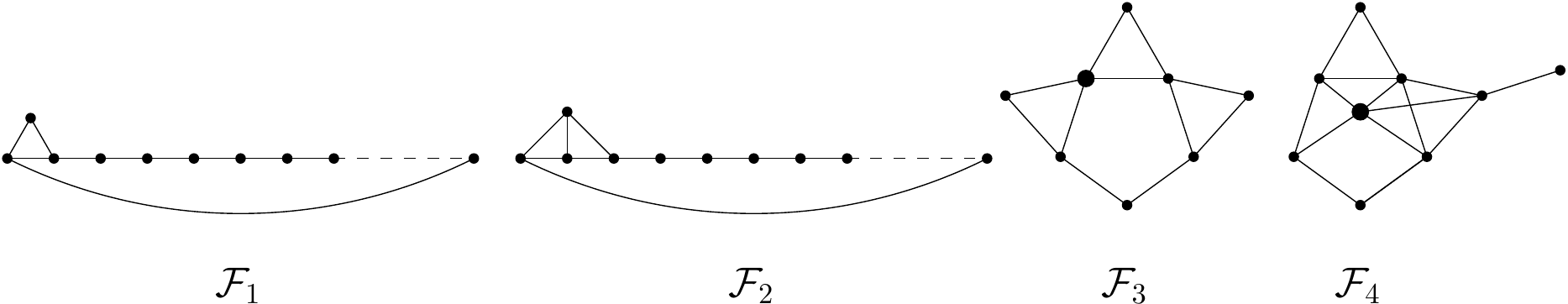}
    \caption{Families $\FF_1, \FF_2, \FF_3, \FF_4$ of graphs.
    The families $\FF_1$ and $\FF_2$ consist of graphs which are obtained
    from $C_{\ell}$ (where $\ell \geq 9$ and $\ell$ is odd)
    by adding a vertex such that the set of its neighbours on $C_{\ell}$
    induces $P_2$, $P_3$, respectively.
    The families $\FF_3$ and $\FF_4$ consist of graphs which are obtained as follows.
    We consider the graph labelled $\FF_3$, $\FF_4$, respectively,
    and its distinguished vertex (depicted as large), 
    and add (in sequence) an arbitrary number of vertices
    adjacent exactly to the distinguished vertex and to its neighbours.}
    \label{figFamilies}
\end{figure}

\begin{figure}[ht]
\centering
\includegraphics[scale=0.7]{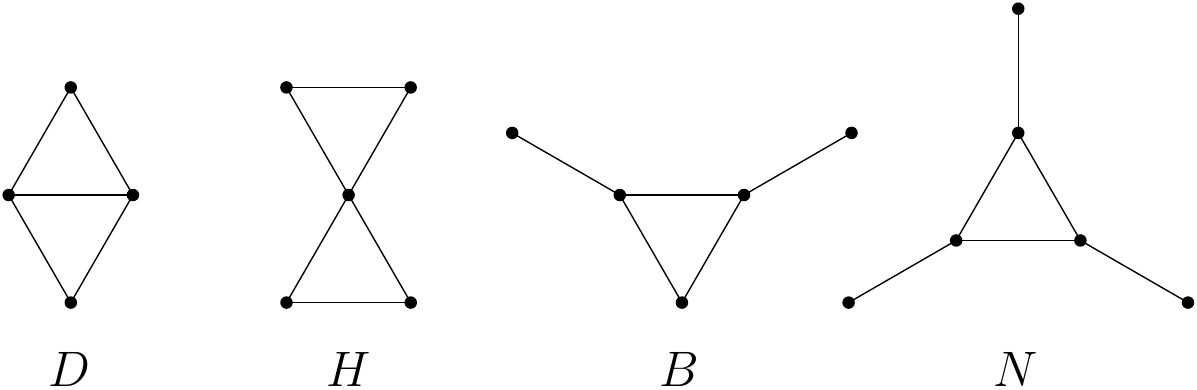}
    \caption{The graphs $D$, $H, B$ and $N$.}
    \label{figDHBN}
\end{figure}

\begin{proof}[Proof of Lemma~\ref{l7}]
We use the assumption that $X$ is not an induced subgraph of any of
$P_6$, $K_1 \cup P_5$, $2P_3$, $Z_2$, $K_1 \cup Z_1$, $2K_1 \cup K_3$,
and we show that $X$ contains at least one of the graphs
$5K_1$, $3K_1 \cup K_2$, $3K_2$, $K_2 \cup P_4$,
$\claw$, $C_4$, $C_5$, $C_6$, $C_7$,
$K_4$, $D$, $H$, $B$, $K_2 \cup K_3$, $K_1 \cup Z_2$ 
as an induced subgraph
(the graphs $D$, $H$ and $B$ are depicted in Figure~\ref{figDHBN}).
Clearly, we can assume that $X$ is a chordal graph
(otherwise it contains at least one of
$C_4$, $C_5$, $C_6$, $C_7$, $K_2 \cup P_4$
as an induced subgraph and the claim is satisfied). 
In addition, we can assume that $X$ contains at most one triangle
(otherwise there is at least one of $K_4$, $D$, $H$, $K_2 \cup K_3$
as an induced subgraph or $X$ is not chordal).
We discuss two cases.

For the case when $X$ contains no triangle,
we can also assume that every component of $X$ is a path
(otherwise there is $\claw$).
We use the assumption that $X$ is not an induced subgraph of any of
$P_6$, $K_1 \cup P_5$, $2P_3$,
and we let $M$ denote a set of all vertices of a largest component of $X$
(that is, vertices of a longest path), 
and we observe the following.
If $|M| \geq 7$, then $X$ contains $K_2 \cup P_4$ as an induced subgraph.
If $|M| = 6$, then $X$ has at least two components
(since $X$ is not $P_6$),
and hence it contains induced $3K_1 \cup K_2$.
If $5 \geq |M| \geq 4$,
then $X$ has at least two vertices which do not belong to $M$
(since $X$ is not an induced subgraph of $K_1 \cup P_5$),
and thus $X$ contains $3K_1 \cup K_2$ or $K_2 \cup P_4$ as an induced subgraph.
If $|M| = 3$,
then $X$ has at least three components and at least three vertices outside $M$
(since $X$ is not an induced subgraph of $2P_3$).
It follows that $X$ contains induced $3K_1 \cup K_2$.
If $|M| = 2$, then we note that $X$ contains $3K_2$ or $3K_1 \cup K_2$
as an induced subgraph.
Lastly if $|M| = 1$, then $X$ has at least five components
and this gives $5K_1$.

For the other case, we consider the component of $X$ which contains the triangle.
We can assume that this component is a subgraph of $Z_2$
(otherwise there is $B$ or $K_2 \cup K_3$ as an induced subgraph),
and that every other component is trivial
(otherwise we have induced $K_2 \cup K_3$).
Since $X$ is not an induced subgraph of any of $Z_2$, $K_1 \cup Z_1$, $2K_1 \cup K_3$,
we conclude that $X$ contains $K_1 \cup Z_2$ or $3K_1 \cup K_2$
is an induced subgraph.

We proceed by considering the families $\FF_1, \FF_2, \FF_3, \FF_4$ of graphs 
depicted in Figure~\ref{figFamilies},
and we note that each of the graphs is $\claw$-free,
of independence number at least $4$,
and contains an induced odd hole.
Furthermore, we observe that every graph of $\FF_1$ is
$\{ C_4, C_5, C_6, C_7, K_4, D, H \}$-free,
and every graph of $\FF_2$ is
$B$-free,
and every graph of $\FF_3$ is
$\{ 5K_1, 3K_1 \cup K_2, K_1 \cup Z_2, 3K_2, K_2 \cup P_4 \}$-free,
and every graph of $\FF_4$ is
$K_2 \cup K_3$-free.
\end{proof}

\section{Concluding remarks}
\label{sec:concl}

Finally, we remark that the class of all
imperfect connected $\{\claw,B_{1,2}\}$-free graphs with 
independence number at least $4$ admits a simple characterisation.
The graph $B_{1,2}$ is depicted in Figure~\ref{figChoicesOfX}
and the characterisation is given in Theorem~\ref{thm-bull-exceptions}.
(A similar, but more technical, structural statement can be shown for 
$X$ chosen as the graph $N$ depicted in Figure~\ref{figDHBN}.)

We start by recalling the notation of an inflation of a cycle
(also used in~\cite{alpha3}).
We say that a graph is an {\it inflation of $C_k$}
if the graph can be obtained from $C_k$
by applying (in sequence) the following operation any number of times (possibly not at all).
Choose an arbitrary vertex of the graph on hand and add
a new vertex adjacent precisely to the chosen vertex and to all its neighbours.
An example of an inflation of $C_7$ is depicted in Figure~\ref{figInflation}.

\begin{figure}[ht]
\centering
\includegraphics[scale=0.7]{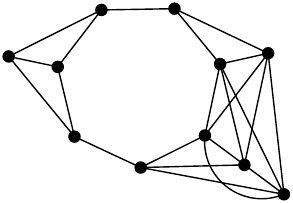}
\caption{An inflation of $C_7$.}
\label{figInflation}
\end{figure}
We show the following fact
(a similar statement considering $\{\claw,B\}$-free graphs
was shown in~\cite{alpha3}).

\begin{theorem}
\label{thm-bull-exceptions}
Let $G$ be a connected $\{\claw,B_{1,2}\}$-free graph with 
independence number at least $4$.
Then $G$ is either perfect or 
it is an inflation of $C_k$ such that $k$ is odd and $k\geq 9$.
\end{theorem}

\begin{proof}
We note that each of the graphs 
$H_1,\ldots,H_7$ (depiced in Figure~\ref{figH}) contains an induced $B_{1,2}$,
and so $G$ is $\{ C_5, C_7 \}$-free by Lemma~\ref{l5}.
Thus, $G$ contains no induced odd antihole by Lemma~\ref{lBR}.
We conclude that the statement follows by the combination of 
Theorem~\ref{tSPGT} and Observation~\ref{prop-bull-cycle-class_C}
(stated below).
\end{proof}
We let $B_{i,j}$ denote the graph obtained by identifying 
end-vertices of two vertex-disjoint paths $P_{i+1}$, $P_{j+1}$ (one end of each)
with two distinct vertices of a triangle
(for instance, see the graph $B_{1,2}$ depicted in Figure~\ref{figChoicesOfX}).
We show the following structural
observation on $\{\claw,B_{1,p}\}$-free graphs.
(The argument goes along similar lines as in
the proof of~\cite[Lemma 4.2]{alpha3}, where 
$\{\claw,\bullp\}$-free 
graphs and $k \geq 5$ were considered.)

\begin{obs}
\label{prop-bull-cycle-class_C}
Let $p$ be an integer greater than $1$ and $G$ be a connected $\{\claw,\bullp\}$-free 
graph which contains an induced $C_k$ such that $k \geq 2p+3$. 
Then $G$ is an inflation of $C_k$.
\end{obs}

\begin{proof} 
We consider a set $C$ inducing $C_k$ in $G$.
Clearly, we can assume that there is a vertex, say $x$, of $V(G)\setminus C$ 
(otherwise, the statement is satisfied trivially),
and that $x$ is adjacent to a vertex of $C$ 
(since $G$ is connected).

We note that $N(x)\cap C$ cannot induce any of the graphs 
$K_2$, $P_4$, $2K_2$ (since $G$ is $\bullp$-free),
and thus it induces $P_3$ (by Observation~\ref{o8}).
In particular, we get that every vertex of $G$ is adjacent to at least one vertex of $C$
(since $G$ is connected and $\claw$-free).

We consider a pair of vertices, say $x$ and $y$, of $V(G)\setminus C$,
and let $c$ be the number of their common neighbours in $C$.
Using that $G$ is $\{\claw,\bullp\}$-free, we discuss the cases
given by $c = 0,1,2,3$ (see Figure~\ref{figXy}),
and we conclude that $x$ and $y$ are adjacent if and only if
$c \geq 2$.
It follows that $G$ is an inflation of $C_k$.
\end{proof}

\begin{figure}[ht]
\centering
\includegraphics[scale=0.7]{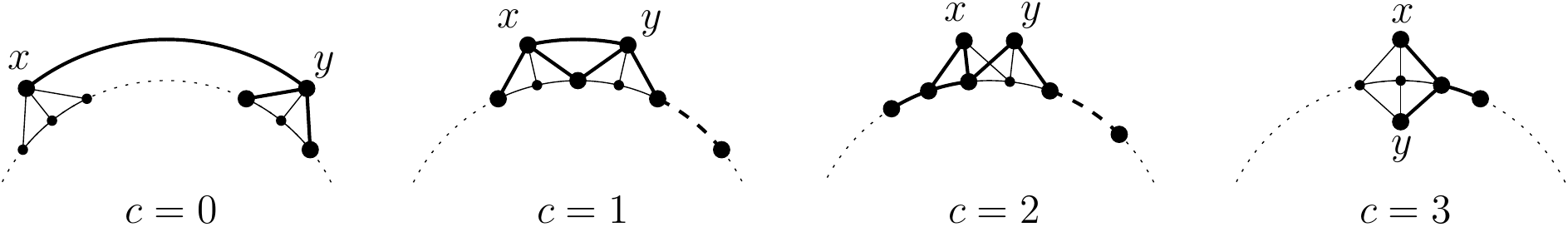}
    \caption{Adjacencies of $x$ and $y$ and $C$ giving one of the forbidden subgraphs
	(induced copies of $\claw$ and $\bullp$ are depicted in bold).}
    \label{figXy}
\end{figure}

\section*{Acknowledgements}
We thank the anonymous referees for their helpful remarks and suggestions.
The research was partly supported by the DAAD-PPP project
`Colourings and connection in graphs'
with project-ID 57210296 (German) and 7AMB16DE001 (Czech), respectively.
The research of the third, fourth, fifth and seventh author was also partly supported by project
GA20-09525S of the Czech Science Foundation.


\begin{thebibliography}{9}
%
\bibitem{BR} A. Ben Rebea:
\'{E}tude des stables dans les graphes quasi-adjoints, Th\`{e}se,
Universit\'{e} de Grenoble, France, 1981.
%
\bibitem{BM08} J. A. Bondy, U.\,S.\,R. Murty: Graph Theory,
Springer, 2008.
%
\bibitem{alpha3}
C. Brause, P. Holub, A. Kabela, Z. Ryj{\'a}\v{c}ek, I. Schiermeyer, P.~Vr{\'a}na:
On forbidden induced subgraphs for $\claw$-free perfect graphs,
Discrete Mathematics 342 (2019), 1602--1608.
%
\bibitem{ChRST06}
M. Chudnovsky, N. Robertson, P. Seymour, R. Thomas: 
The strong perfect graph theorem,
Annals of Mathematics 164 (2006), 51--229.
%
\bibitem{CS}
M. Chudnovsky and P. Seymour:
Claw-free graphs VI. Colouring,
Journal of Combinatorial Theory, Series B 100 (2010), 560--572.
%
\bibitem{ChS88}
V. Chv\'atal, N. Sbihi: Recognizing claw-free perfect graphs,
Journal of Combinatorial Theory, Series B 44 (1988), 154--176.
%
\bibitem{Fo93}
J. L. Fouquet: A strengthening of Ben Rebea's lemma, 
Journal of Combinatorial Theory, Series B 59 (1993), 35--40.
%
\bibitem{GJPS17}
P. A. Golovach, M. Johnson, D. Paulusma, J. Song:
A survey on the computational complexity of coloring graphs with forbidden subgraphs,
Journal of Graph Theory 84 (2017), 331--363. 
%
\bibitem{Kim}
J. H. Kim:
The Ramsey number $R(3,t)$ has order of magnitude $\frac{t^2}{log t}$, 
Random Structures and Algorithms 7 (1995), 173--207.
%
\bibitem{RS04}
B. Randerath, I. Schiermeyer:
Vertex colouring and forbidden subgraphs - a survey,
Graphs and Combinatorics 20 (2004), 1--40. 
%
\bibitem{RS19}
B. Randerath, I. Schiermeyer: 
Polynomial $\chi$-binding functions and forbidden induced subgraphs: a survey,
Graphs and Combinatorics 35 (2019), 1--31.
%
\end{thebibliography}
\end{document}